\documentclass[a4paper,12pt]{article}
\usepackage[utf8]{inputenc}
\usepackage{amsfonts,amssymb,amsmath,amsthm}
\usepackage{textcomp}
\usepackage[scr=boondoxo]{mathalfa}
\usepackage{enumitem}
\usepackage[all]{xy}
\usepackage{tikz}
\usepackage{mathtools}
\usepackage{color}
\usepackage[unicode]{hyperref}

\hypersetup{
    colorlinks,
    linkcolor={blue!80!black},
    citecolor={blue!80!black},
    urlcolor={blue!80!black}
}

\DeclareFontFamily{U}{mathx}{\hyphenchar\font45}
\DeclareFontShape{U}{mathx}{m}{n}{
      <5> <6> <7> <8> <9> <10>
      <10.95> <12> <14.4> <17.28> <20.74> <24.88>
      mathx10
      }{}
\DeclareSymbolFont{mathx}{U}{mathx}{m}{n}
\DeclareFontSubstitution{U}{mathx}{m}{n}
\DeclareMathAccent{\widecheck}{0}{mathx}{"71}
\DeclareMathAccent{\wideparen}{0}{mathx}{"75}

\newcommand{\coleq}{\coloneqq}
\newcommand{\eqcol}{\eqqcolon}
\newcommand{\bine}{\mathbin{\varepsilon}}
\newcommand{\e}{\varepsilon}

\newcommand{\Nu}{\mathscr{N}}
\newcommand{\Co}{\mathscr{C\!o}}

\theoremstyle{plain}
\newtheorem{theorem}{Theorem}
\newtheorem*{definition*}{Definition}
\newtheorem{proposition}[theorem]{Proposition}
\newtheorem{lemma}[theorem]{Lemma}

\newtheorem*{remark*}{Remark}
\DeclareMathOperator{\supp}{supp}
\DeclareMathOperator{\can}{can}
\DeclareMathOperator{\Id}{Id}

\newcommand{\cB}{{\mathcal{B}}}
\newcommand{\cC}{\mathcal{C}}
\newcommand{\cD}{\mathcal{D}}
\newcommand{\cE}{\mathcal{E}}
\newcommand{\cH}{\mathcal{H}}
\newcommand{\cL}{\mathcal{L}}

\newcommand{\cN}{\mathcal{N}}
\newcommand{\cO}{\mathcal{O}}
\newcommand{\cS}{\mathcal{S}}
\newcommand{\bN}{\mathbb{N}}
\newcommand{\bR}{\mathbb{R}}
\newcommand{\bC}{\mathbb{C}}
\newcommand{\pd}{\partial}
\newcommand{\ud}{\mathrm{d}}

\newcommand{\abso}[1]{\left|#1\right|}
\newcommand{\norm}[1]{\left\lVert#1\right\rVert}

\title{The space $\dot\cB'$ of distributions vanishing at infinity -- duals of tensor products}
\author{E.~A.~Nigsch\footnote{Wolfgang-Pauli-Institut, Oskar-Morgenstern-Platz 1, 1090 Wien, Austria. e-mail: \href{mailto:eduard.nigsch@univie.ac.at}{eduard.nigsch@univie.ac.at}. Corresponding author.}, N. Ortner\footnote{Institut für Mathematik, Universität Innsbruck, Technikerstraße 13, 6020 Innsbruck, Austria.}}
\date{April 11, 2016}

\begin{document}

\maketitle

\begin{abstract}
Analogous to L.~Schwartz' study of the space $\cD'(\cE)$ of semi-regular distributions we investigate the topological properties of the space $\cD'(\dot\cB)$ of semi-regular vanishing distributions and give representations of its dual and of the scalar product with this dual. In order to determine the dual of the space of semi-regular vanishing distributions we generalize and modify a result of A.\ Grothendieck on the duals of $E \widehat \otimes F$ if $E$ and $F$ are quasi-complete and $F$ is not necessarily semi-reflexive.
\end{abstract}

{\bfseries Keywords: }semi-regular vanishing distributions, duals of tensor products

{\bfseries MSC2010 Classification: 46A32, 46F05}

\section{Introduction}\label{sec0}

L.~Schwartz investigated, in his theory of vector-valued distributions \cite{zbMATH03145498,zbMATH03145499}, several subspaces of the space $\cD'_{xy} = \cD'(\bR^n_x \times \bR^m_y)$ that are of the type $\cD'_x(E_y)$ where $E_y$ is a distribution space.

Three prominent examples are:
\begin{enumerate}[label=(\roman*)]
 \item $\cD'_x(\cS_y')$ -- the space of semi-temperate distributions, for which the partial Fourier transform is defined \cite[p.~123]{zbMATH03145498};
 \item $\cD'_x(\cD'_{L^1,y})$ -- the space of partially summable distributions, for which the convolution of kernels is defined \cite[\S 2, p.~546]{zbMATH03163214}, as a generalization of the convolvability condition for two distributions \cite[pp.~131,132]{zbMATH03145498};
 \item $\cD'_x(\cE_y)$ -- the space of semi-regular distributions (see \cite{MR0090777} and \cite[p.~99]{zbMATH03145498}).
\end{enumerate}

In this paper we will be concerned with the space $\cD'(\dot\cB)$ of ``semi-regular vanishing'' distributions.

{\bfseries Notation and Preliminaries.} We will mostly build on notions from \cite{TD,zbMATH03199982,zbMATH03145498, zbMATH03145499}.
$\cD'(E)$ is defined as $\cD' \bine E$, which space coincides with $\cD' \widehat\otimes_\e E = \cD' \widehat\otimes_\pi E \eqcol \cD' \widehat\otimes E$ in the 3 examples above. If $E,F$ are two locally convex spaces then $E \widehat\otimes_\pi F$, $E \widehat\otimes_\e F$ and $E \widehat\otimes_\iota F$ denote the completion of their projective, injective, and inductive tensor product, respectively; writing $\wideparen\otimes$ in place of $\widehat\otimes$ means that we take the quasi-completion instead. The subscript $\beta$ in $E \otimes_\beta F$ \cite[p.~12, 2°]{zbMATH03145499} refers to the finest locally convex topology on $E \otimes F$ for which the canonical injection $E \times F \to E \otimes F$ is hypocontinuous with respect to bounded sets.
Given a locally convex space $E$, $E_b'$ denotes its strong dual, $E_\sigma'$ its weak dual and $E_c'$ its dual with the topology of uniform convergence on absolutely convex compact sets. In absence of any of these designations, $E'$ carries the strong dual topology. For the definition of $\overline{D}(E)$ see \cite[p.~63]{zbMATH03145498} and \cite[p.~94]{FDVV}. $\dot\cB'$ is the space of distributions vanishing at infinity, i.e., the closure of $\cE'$ in $\cD'_{L^\infty}$ \cite[p.~200]{TD}.
$\Nu(E,F)$ and $\Co(E,F)$ denote the space of nuclear and compact linear operators $E \to F$, respectively. The normed space $E_U$ for an absolutely convex zero-neighborhood $U$ in $E$ is introduced in \cite[Chap.~I, p.~7]{zbMATH03199982}, with associated canonical mapping $\Phi_U \colon E \to \widehat E_U$. $\cL(E,F)$ is the space of continuous linear mappings $E \to F$. By $\cB^s(E,F)$ and $\cB(E,F)$ we denote the spaces of separately continuous and continuous bilinear forms $E \times F \to \bC$, respectively, and $\cB^h(E,F)$ is the space of separately hypocontinuous bilinear forms; for any of these spaces, the index $\e$ denotes the topology of bi-equicontinous convergence.

{\bfseries Motivation.} In order to prove, e.g., the equivalence of $S(x-y)T(y) \in \cD'_x(\cD'_{L^1,y})$ for two distributions $S,T$ (which means, by definition, that $S,T \in \cD'(\bR^n)$ are convolvable) to the inequality
\begin{multline*}
\forall K \subset \bR^{n} \textrm{ compact }\exists C>0\ \exists m \in \bN_0\\
\forall \varphi \in \cD(\bR^{2n}) \textrm{ with }\supp \varphi \subseteq \{ (x,y) \in \bR^{2n}: x+y \in K\}:\\
\abso{\langle \varphi(x, y), S(x)T(y) \rangle } \le C \sup_{\abso{\alpha} \le m} \norm{ \pd^\alpha \varphi }_\infty,
\end{multline*}
it is advantagous to know of a ``predual'' of $\cD'(\cD'_{L^1})$, i.e., the space $\overline{\cD}\ ( \dot\cB) = \varinjlim_K ( \cD_K \widehat\otimes \dot\cB ) = \cD \widehat\otimes_\iota \dot\cB$ for which $ ( \overline{\cD} ( \dot\cB ) )'_b = \cD' \widehat\otimes \cD'_{L^1}$ \cite[Prop.~3, p.~541]{zbMATH03163214}.
A ``predual'' of $\cD'(\cS')$ and $\cD'(\cD'_{L^1})$ can easily be found by Corollary 3 in \cite[p.~104]{zbMATH03145498}, which states that $(\overline{\cD}(E))'_b = \cD'(E')$ if $E$ is a Fr\'echet space. A ``predual'' of $\cD'(\cE)$ is the space $\cD \widehat\otimes_\iota \cE'$ (see Propositions \ref{prop1} and \ref{prop2}).

In the memoir \cite{MR0090777} L.~Schwartz investigated the space $\cD'(\cE)$ of semi-regular distributions. For reasons of comparison we present the main features thereof in Section \ref{sec1}, i.e., in Proposition \ref{prop1} properties of $\cD'(\cE)$, in Proposition \ref{prop2} the dual and a ``predual'' of $\cD'(\cE)$ and in Proposition \ref{prop3} an explicit expression for the scalar product of $K(x,y) \in \cD'_x(\cE_y)$ with $L(x,y) \in \cD_x \widehat\otimes_\iota \cE'_y$. These propositions generalize the corresponding propositions in \cite{MR0090777} and new proofs are given.

In \cite{MR573803} we found the condition
\begin{equation}\label{eq1}
 \forall \varphi \in \cD: (\varphi * \check S) T \in \dot\cB'
\end{equation}
for two distributions $S,T$, in order that $(\pd_j S ) * T = S * (\pd_j T)$ under the assumption that $(\pd_j S, T)$ and $(S, \pd_j T)$ are convolvable (see also \cite[p.~559]{HorvathSI}). The equivalence of \eqref{eq1} and
\begin{equation}\label{eq2}
S(x-y)T(y) \in \cD'_x(\dot\cB'_y)
\end{equation}
is proven in \cite{zbMATH05770245}. Due to the regularization property
\[ S \in \dot\cB' \Longleftrightarrow S(x-y) \in \cD'_y ( \dot\cB_x ) \]
for a distribution S \cite[remarque 3°, p.~202]{TD} we are motivated to investigate the space $\cD'(\dot\cB)$ of ``semi-regular vanishing distributions'' analogously to $\cD'(\cE)$ in \cite{MR0090777}, i.e.,
\begin{itemize}
 \item to state properties of $\cD'(\dot\cB)$ in Proposition \ref{prop4},
 \item to determine the dual of $\cD'(\dot\cB)$ in Proposition \ref{prop5},
 \item to express explicitly the scalar product in Proposition \ref{prop6}, and
 \item to determine the transpose of the regularization mapping $\dot\cB' \to \cD'_y ( \dot\cB_x)$, $S \mapsto S(x-y)$ in Proposition \ref{prop7}.
\end{itemize}

{\bfseries Duals of tensor products.} Looking for $(\cD'(\dot\cB))'_b$ we make use of the following duality result of A.~Grothendieck which allows -- in contrast to the corresponding propositions in \cite[\S 45, 3.(1), p.~301, \S 45, 3.(5), p.~302, \S 45, 3.(7), p.~304]{Koethe2}, \cite[16.1.7, p.~346]{Jarchow}, \cite[IV, 9.9, p.~175; Corollary 1, p.~176; Ex.~32, pp.~198, 199]{Schaefer} and \cite[\S 4, Satz 1, p.~212]{zbMATH03498641} -- to determine the duals of tensor products $E \otimes F$ in cases where $E$ and $F$ are of ``different nature'':

\begin{quote}
\cite[Chap.~II, \S 4, n°1, Lemme 9, Corollaire, p.~90]{zbMATH03199982}: Let $E$ and $F$ be complete locally convex spaces, $E$ nuclear, $F$ semireflexive. If the strong dual $(E \widehat\otimes F)'_b$ is complete then $(E \widehat\otimes F)'_b = E'_b \widehat\otimes_\iota F'_b$.
\end{quote}

Note that for our example in the introduction the assumption of semireflexivity is not fulfilled. Nevertheless we reach the conclusion by observing that
\[ (\cD'(\dot\cB))'_b = (\cD'(\cB_c))'_b, \]
$\cB_c$ being the semireflexive space $\cD_{L^\infty}$ endowed with the topology of uniform convergence on compact subsets of $\cD'_{L^1}$ \cite[p.~203]{TD} which also can be described by seminorms \cite[Prop.~1.3.1, p.~11]{zbMATH06308371}. Therefore, in Section \ref{sec3}, we prove a generalization of Grothendieck's Corollary and a modification which applies to semi-reflexive locally convex Hausdorff spaces $F$ such that the completeness of $(E \widehat\otimes F)'_b$ can be shown by the existence of a space $F_0$ such that $(E \widehat\otimes F_0)'_b$ is complete and $(E \widehat\otimes F)'_b = (E \widehat\otimes F_0)'_b$.

{\bfseries The kernel identity for $\dot\cB'$.} There is yet another condition equivalent to \eqref{eq1} and \eqref{eq2}:
\begin{equation}\label{eq3}
 \delta(z-x-y)S(x)T(y) \in \cD'_z \widehat\otimes \dot\cB'_{xy}
\end{equation}
which is nothing else than
\begin{multline*}
 \forall K \subset \bR^{n} \textrm{ compact } \exists C>0\ \exists m \in \bN_0\\
\forall \varphi \in \cD(\bR^{2n}) \textrm{ with }\supp \varphi \subseteq \{ (x,y) \in \bR^{2n}: x+y \in K\}: \\
 \abso{ \langle \varphi (x, y), S(x)T(y) \rangle } \le C \sup_{\abso{\alpha} \le m} \norm { \pd^\alpha \varphi }_1.
\end{multline*}
This equivalence can be shown by the use of the ``kernel identity''
\[ \dot\cB'_{xy} = \dot\cB'_x \widehat\otimes_\e \dot\cB_y' = \dot\cB'_x ( \dot\cB'_y) \]
which we prove in Section \ref{sec4} (Proposition \ref{prop10}). For similar identities see \cite[Prop.~17, p.~59; Prop.~28, p.~98]{zbMATH03145498} and \cite[Chap.~I, Cor.~4, p.~61; Ex., p.~90]{zbMATH03199982}.

A preliminary version of this paper was presented in two talks in Vienna and in Innsbruck by N.~O.

\section{Semi-regular distributions.}\label{sec1}

Whereas L.~Schwartz stated that the space $\cD'(\cE)$ is semireflexive \cite[p.~110]{MR0090777} we prove

\begin{proposition}[Properties of $\cD'(\cE)$]\label{prop1}
The space of semi-regular distributions $\cD'(\cE)$ is
\begin{enumerate}[label=(\roman*)]
 \item\label{prop1.1} nuclear,
 \item\label{prop1.2} ultrabornological and
 \item\label{prop1.3} reflexive.
\end{enumerate}
\end{proposition}
\begin{proof}
 \ref{prop1.1} The nuclearity follows by Grothendieck's permanence result in \cite[Chap.~II, \S 2, n° 2, Th\'eor\`eme 9, 3°, p.~47]{zbMATH03199982} (see also \cite[p.~110]{MR0090777}).

\ref{prop1.2} M.~Valdivia's sequence space repesentation of $\cO_M$ \cite[Theorem 3, p.~478]{zbMATH03740168} and \cite[Chap.~I, \S 1, n°3, Proposition 6, 1°, p.~46]{zbMATH03199982} yield the isomorphisms
\[ \cD' \widehat\otimes \cE \cong s'^{\bN} \widehat\otimes s^{\bN} \cong (s' \widehat \otimes s)^\bN \cong \cO_M^{\bN}. \]
By \cite[Chap.~II, \S 4, n°4, Th\'eor\`eme 16, p.~131]{zbMATH03199982}, $\cO_M$ is ultrabornological and, hence, also $\cO_M^\bN$, by \cite[Folgerung 5.3.8, p.~106]{zbMATH03783682} (see also \cite[13.5.3, p.~281]{Jarchow} or \cite[\S 28, 8.(6), p.~392]{Koethe1}).

\ref{prop1.3} The semi-reflexivity of $\cD'(\cE)$ by \cite[Chap.~II, \S 3, n°2, Proposition 13, p.~76]{zbMATH03199982} (see also \cite[p.~110]{MR0090777}) and the barrelledness of $\cD'(\cE)$ by \ref{prop1.2} yield the reflexivity of $\cD'(\cE)$.
\end{proof}

In \cite{MR0090777}, the dual of $\cD'(\cE)$ is described by the representation of its elements as finite sums of derivatives with respect to $y$ of functions $g(x,y) \in \cD_x \widehat\otimes \cD^0_{y,c}$ \cite[Proposition 1, p.~112]{MR0090777}. Thus, $(\cD'_x ( \cE_y))' = \varinjlim_m ( \cD_x \widehat\otimes \cE_{c,y}^{\prime m})$ algebraically.
Note that L.\ Schwarz asserts on the one hand $(\cD'_x \widehat\otimes \cE_y)' = \varinjlim_m ( \cD_x \widehat\otimes \cE^{\prime m}_y)$ in \cite[Prop.~1, p.~112]{MR0090777}, whereas we find, on the other hand, $(\cD'_x \widehat\otimes \cE_y)' = \varinjlim_m ( \cD_x \widehat\otimes \cE^{\prime m}_{c,y} )$ in \cite[Corollaire 1, p.~116]{MR0090777}. It can be shown that the isomorphism is also a topological one. Other representations of the dual of $\cD'(\cE)$ are given in:

\begin{proposition}[Dual of $\cD'(\cE))$]\label{prop2}
We have linear topological isomorphisms
\begin{align}
 \label{prop2.eq1} (\cD'(\cE))'_b & \cong \cD \widehat\otimes_\iota \cE' = \cD \widehat\otimes_\beta \cE' \\
\intertext{and linear isomorphisms}
\label{prop2.eq2}(\cD'(\cE))' & \cong \cD(\cE'; \e ) = \cD(\cE'; \beta) \\
\label{prop2.eq3}& \cong \cE'(\cD; \e) = \cE'(\cD; \beta) \\
\nonumber & \cong \Nu(\cD', \cE') = \Co ( \cD', \cE') \\
\nonumber & \cong \Nu ( \cE, \cD ) = \Co ( \cE, \cD).
\end{align}
\end{proposition}

For the notation $\cD(\cE'; \e)$ or $\cD(\cE'; \beta)$ see \cite[p.~54]{zbMATH03145499}.

\begin{proof}
The first isomorphism of \eqref{prop2.eq1} results from the Corollary cited in the introduction: the 5 hypotheses are fulfilled due to Proposition \ref{prop1} \ref{prop1.2}. The equality $E \otimes_\beta F = E \otimes_\iota F$ is a consequence of the barrelledness of $E$ and $F$, in the case above: $E = \cD$, $F = \cE'$.

The isomorphisms \eqref{prop2.eq2} and \eqref{prop2.eq3} follow by \cite[Proposition 22, p.~103]{zbMATH03145499}.

The coincidence of nuclear and compact linear operators is a consequence of \cite[Chap.~II, \S 2, n°1, Corollaire 4, 1°, p.~39]{zbMATH03199982} and the semi-Montel property of $\cD$ and $\cE'$.
\end{proof}

\begin{proposition}[Existence and uniqueness of the scalar product]\label{prop3}
There is one and only one scalar product
\[
  \genfrac{}{}{0pt}{0}{ \langle \ , \ \rangle_x}{\langle \ , \ \rangle_y}
  \colon ( \cE'_x \widehat\otimes_\beta \cD_y ) \times ( \cE_x \widehat\otimes \cD'_y ) \to \bC
\]
which is partially continuous and coincides on $(\cE'_x \otimes \cD_y) \times (\cE_x \otimes \cD'_y)$ with the product
$ {}_{\cE_x} \langle \ , \ \rangle_{\cE_x'} \cdot {}_{\cD_y} \langle \ , \ \rangle_{\cD'_y}$, i.e., for decomposed elements $(S(x) \otimes \varphi(y), \psi(x) \otimes T(y))$ we have
\[ \genfrac{}{}{0pt}{0}{ \langle \ , \ \rangle_x}{\langle \ , \ \rangle_y} \bigl(S(x) \otimes \varphi(y), \psi(x) \otimes T(y)\bigr) = {}_{\cE_x} \bigl\langle \psi(x), S(x) \bigr\rangle_{\cE'_x} \cdot {}_{\cD_y} \bigl\langle \varphi(y), T(y) \bigr\rangle_{\cD'_y}. \]
\end{proposition}

\begin{proof}
 A \emph{first} proof is given in \cite[Proposition 1, p.~112]{MR0090777}, by means of the explicit representation of the elements of the strong dual $\cD \widehat\otimes_\iota \cE'$ of $\cD'(\cE)$ hinted at before Proposition \ref{prop2}. The uniqueness can be presumed because L.~Schwartz uses the word ``\emph{le} produit scalaire''.

A \emph{second} proof consists in applying the ``Th\'eor\`emes de croisement'', i.e., \cite[Proposition 2, p.~18]{zbMATH03145499}: existence, uniqueness and partial continuity of the scalar product follow.

A \emph{third} proof follows by composition of the vectorial scalar product
\[ \langle\ ,\ \rangle \colon \cE_x ( \cD'_y ) \times (\cE'_x \widehat\otimes_\beta \cD_y) \to \cD'_y \widehat\otimes_\iota \cD_y = \cD'_y \widehat\otimes_\beta \cD_y \]
\cite[Proposition 10, p.~57]{zbMATH03145499} with the scalar product
\[ \langle\ ,\ \rangle_y \colon \cD'_y \widehat\otimes_\iota \cD_y \to \bC.  \qedhere\]
\end{proof}

\section{\texorpdfstring{``Semi-regular vanishing''}{'Semi-regular vanishing'} distributions}\label{sec2}

\begin{proposition}[Properties of $\cD' \widehat\otimes \dot\cB$ and $\cD' \widehat\otimes \cB_c$]\label{prop4}
The space of ``semi-regular vanishing'' distributions $\cD'(\dot\cB) = \cD' \widehat\otimes \dot\cB$ is ultrabornological but not semi-reflexive. $\cD'(\cB_c) = \cD' \widehat\otimes \cB_c$ is semireflexive but not bornological.
\end{proposition}

\begin{proof}
 The non-semireflexivity and the semireflexivity, respectively, are consequences of Grothendieck's permanence result \cite[Chap.~II, \S 3, n°2, Proposition 13 e., p.~76]{zbMATH03199982} due to the corresponding properties of $\dot{\cB}$ and $\cB_c$, respectively.

The sequence space representations
\[ \dot\cB \cong c_0 \widehat\otimes s, \qquad \cO_M \cong s \widehat\otimes s',\qquad \cD' \cong {s'}^{\bN} \]
(see \cite[Theorem 3.2, p.~415; Theorem 5.3, p.~427]{zbMATH03822460} and \cite[Theorem 3, p.~478]{zbMATH03740168})
show that
\[ \cD' \widehat\otimes \dot\cB \cong {s'}^{\bN} \widehat\otimes ( s \widehat\otimes c_0 ) \cong \cO_M^\bN \widehat \otimes c_0. \]
$\cO_M^\bN$ is ultrabornological (as seen in the proof of Proposition \ref{prop1} \ref{prop1.2}), nuclear, and by \cite[Chap.~II, \S 2, n°2, Th\'eor\`eme 9, 2°, p.~47]{zbMATH03199982} also its dual $\cO_M^{\prime(\bN)}$ is nuclear, such that the bornologicity of $\cO_M^\bN \widehat\otimes c_0$ follows by \cite[Proposition 2, p.~75]{zbMATH03312960}.

To show that $\cD'(\cB_c)$ is not bornological we have to find a linear form $K$ in $(\cD' \widehat\otimes \cB_c)^*$, the algebraic dual of $\cD' \widehat\otimes \cB_c$, which is locally bounded (i.e., it maps bounded subsets of $\cD' \widehat\otimes \cB_c$ into bounded sets of complex numbers) but not continuous. Because $\cB_c$ is not bornological there exists $T \in (\cB_c)^*$ which is locally bounded but such that $T \not\in (\cB_c)'$. Fixing any $\varphi_0 \in \cD$ with $\varphi_0(0)=1$, we define $K$ by $K(u) \coleq T ( u ( \varphi_0 ))$ for $u \in \cL(\cD, \cB_c)$. Then $K$ is locally bounded but not continuous. In fact, taking a net $(f_\nu)_\nu \to 0$ in $\cB_c$ such that $T(f_\nu)$ does not converge to zero, we define a net $(u_\nu)_\nu$ in $\cL_b(\cD, \cB_c)$ by $u_\nu(\varphi) \coleq \varphi(0) f_\nu$. Then $u_\nu \to 0$, but $K(u) = T ( u_\nu ( \varphi_0 ) ) = T(f_\nu)$ does not converge to zero.
\end{proof}

Analogously to the explicit description of the elements in $(\cD'(\cE))'_b$, cited before Proposition \ref{prop2}, let us represent the elements of $(\cD'(\dot\cB))'$:

\begingroup
\def\thetheorem{\ref{prop5}$'$}
\addtocounter{theorem}{-1}
\begin{proposition}[Dual of $\cD'(\dot\cB)$]\label{prop5bis}
If $K(x,y) \in \cD'_{xy}$ we have the characterization
\begin{align*}
 K(x,y) \in ( \cD'_x ( \dot\cB_y))' \Longleftrightarrow\ & \exists m \in \bN_0\ \exists g_\alpha(x,y) \in \cD_x \widehat\otimes L_{y}^1,\\
& \abso{\alpha} \le m, \alpha \in \bN_0^n, \textrm{ such that }\\
& K(x,y) = \sum_{\abso{\alpha} \le m} \pd_y^\alpha g_\alpha (x,y),
\end{align*}
i.e., $(\cD'(\dot\cB))' = \varinjlim_m ( \cD \widehat\otimes \cD^{\prime m}_{L^1})$ algebraically.

Furthermore, $(\cD' \widehat\otimes \dot\cB)' = (\cD' \widehat\otimes \cB_c)'$ algebraically.
\end{proposition}
\endgroup
$\cD^{\prime m}_{L^1}$ is the strong dual of the Banach space $\dot\cB^m$ \cite[p.~99]{FDVV}. Note that $\cD \widehat\otimes E = \cD \widehat\otimes_\iota E$ for a Banach space $E$ because separately continuous bilinear forms on $\cD \times E$ are continuous.

\begin{proof}[Proof of Proposition \ref{prop5}$'$]
The following proof is a copy of the proof of \cite[Proposition 1, p.~112]{MR0090777}.

``$\Longrightarrow$'': If $K(x,y) \in ( \cD'_x ( \dot\cB_y ))' = \Nu ( \cD'_x, \cD'_{L^1,y})$ (for the equality, see \cite[Prop. 22, p.~103]{zbMATH03145499} and \cite[Chap.~II, \S 2, n°1, Corollaire 4 1., p.~39]{zbMATH03199982}) there exist a bounded sequence $(\varphi_\nu)_{\nu \in \bN}$ in $\cD_x$, a bounded sequence $(T_\nu)_{\nu \in \bN}$ in $\cD'_{L^1,y}$ and a sequence $(\lambda_\nu)_{\nu \in \bN} \in \ell^1$ such that we have for $S \in \cD'_x$
\[ \langle K(x,y), S(x) \rangle = \sum_\nu \lambda_\nu \langle \varphi_\nu(x), S(x) \rangle T_\nu(y). \]
The boundedness of $(T_\nu)_{\nu \in \bN}$ implies
\begin{multline*}
 \exists m \in \bN_0\ \exists D>0\ \exists f_{\nu,\alpha} \in L^1,\nu\ \in \bN,\alpha \in \bN_0^n, \abso{\alpha} \le m,\\
\textrm{with }\norm{f_{\nu,\alpha}}_1 \le D\textrm{ such that } T_\nu = \sum_{\abso{\alpha} \le m} \pd^\alpha f_{\nu, \alpha}
\end{multline*}
\cite[Remarque 2°, p.~202]{TD} and, hence,
\begin{equation}\label{eqblah}
\langle K(x,y), S(x) \rangle = \sum_{\abso{\alpha} \le m} \pd_y^\alpha \sum_\nu \lambda_\nu \langle \varphi_\nu, S \rangle f_{\nu, \alpha}(y)
\end{equation}
and $K = \sum_{\abso{\alpha} \le m} \pd_y^\alpha g_\alpha$ if we set $g_\alpha (x,y) \coleq \sum_\nu \lambda_\nu \varphi_\nu(x) \cdot f_{\nu, \alpha}(y)$.
In order to see that $g_\alpha(x,y) \in \cD_x \widehat\otimes L^1_y$ it suffices to show that for $S(x) \in \cD'_x$ the sequence of partial sums
\[ \sum_{\nu=1}^N \langle \varphi_\nu(x), S(x) \rangle f_{\nu, \alpha}(y) ,\quad N \in \bN, \]
converges in $L^1$ and that \[ \cD'_x \to L^1_y,\quad S(x) \mapsto \sum_{\nu=1}^\infty \lambda_\nu \langle \varphi_\nu(x), S(x) \rangle f_{\nu, \alpha}(y) \] maps bounded sets of $\cD'$ into bounded sets of $L^1$. This follows from the boundedness of $(\varphi_\nu)$ in $\cD$.

``$\Longleftarrow$'': If $K(x,y) \in \cD'_{xy}$ has the representation
\[ K(x,y) = \sum_{\abso{\alpha} \le m} \pd_y^\alpha g_\alpha(x,y),\quad g_\alpha(x,y) \in \cD_x \widehat\otimes L^1_y, \]
then also $g_\alpha(x,y) \in \cE_x \widehat\otimes L^1_y$.

By \cite[Chap.~I, \S 2, n°1, Th\'eor\`eme 1, 1°, p.~51]{zbMATH03199982} there exists $(\lambda_{\nu, \alpha})_\nu \in \ell^1$, a bounded sequence $(e_{\nu,\alpha})_\nu$ in $\cE_x$ and a bounded sequence $(f_{\nu,\alpha})_\nu \in L^1$ such that
\[ g_\alpha (x,y) = \sum_{\nu} \lambda_\nu f_{\nu,\alpha}(y) e_{\nu_\alpha}(x). \] The compactness of the supports of $g_\alpha$ with respect to $x$ \cite[p.~62]{zbMATH03145498} implies the existence of a function $\phi \in \cD_x$ such that
\[ g_\alpha(x,y) = \phi(x) \cdot g_\alpha(x,y),\ \abso{\alpha} \le m. \]
Thus,
\[ g_\alpha(x,y) = \sum_\nu \lambda_\nu e_{\nu,\alpha} (x) \phi(x) \otimes f_{\nu,\alpha}(y) \]
and
\[ K(x,y) = \sum_\nu \sum_{\abso{\alpha} \le m} \lambda_\nu e_{\nu,\alpha}(x) \phi(x) \otimes \pd_y^\alpha f_{\nu,\alpha}(y). \]
Because $(e_{\nu_\alpha}(x)\phi(x))_{\nu,\alpha}$ is bounded (and hence equicontinuous) in $\cD_x$ and $(\pd_y^\alpha f_{\nu,\alpha}(y))_{\abso{\alpha} \le m; \nu}$ is an equicontinuous (i.e., bounded) subset of $\cD'_{L^1,y}$ this proves that $K \in \cN ( \cD'_x, \cD'_{L^1,y})$.

In order to see that $(\cD' \widehat\otimes \dot\cB)' \subseteq (\cD' \widehat\otimes \cB_c)'$ (the converse inclusion is obvious) we return to equality \eqref{eqblah}. Then for each $f_\alpha \coleq \sum_\nu \abso{\lambda_\nu} \cdot \abso{f_{\nu,\alpha}} \in L^1$ there exists by \cite[Prop.~1.2.1, p.~6]{zbMATH06308371} a function $g_\alpha \in \cC_0$, $g_\alpha > 0$ such that $f_\alpha / g_\alpha \in L^1$. Then for any element $S$ of the polar $U \coleq \{ \varphi_\nu \}^\circ$, which is a $0$-neighborhood in $\cD'$, and any $f \in \dot\cB$ we see that
\begin{align*}
 \langle \langle K(x,y), S(x) \rangle, f(y) \rangle &= \langle \sum_{\abso{\alpha} \le m} \pd_y^\alpha \sum_\nu \lambda_\nu \langle \varphi_\nu, S \rangle f_{\nu, \alpha}(y), f(y) \rangle \\
&= \sum_{\abso{\alpha} \le m} (-1)^{\abso{\alpha}} \langle \sum_\nu \lambda_\nu \langle \varphi_\nu, S \rangle f_{\nu,\alpha}(y), \pd^\alpha f(y) \rangle \\
&= \sum_{\abso{\alpha} \le m} (-1)^{\abso{\alpha}} \int \sum_\nu \lambda_\nu \langle \varphi_\nu, S \rangle f_{\nu, \alpha}(y) \pd^\alpha f(y) \,\ud y
\end{align*}
and thus
\begin{align*}
\abso{ \langle \langle K(x,y), S(x) \rangle, f(y) \rangle } & \le \sum_{\abso{\alpha} \le m} \int \abso{g_\alpha(y) \pd^\alpha f(y)} \frac{ f_\alpha(y) }{g_\alpha(y)} \,\ud y \\
& \le \sum_{\abso{\alpha} \le m} \norm { g_\alpha \cdot \pd^\alpha f}_\infty \int \frac{f_\alpha(y)}{g_\alpha(y)} \,\ud y.
\end{align*}
Hence, $\{ \langle K(x,y), S(x) \rangle : S \in U \} \subseteq (\cB_c)'$ is equicontinuous, which implies the claim.
\end{proof}

\begin{proposition}[Dual of $\cD' \widehat\otimes \dot\cB$]\label{prop5}
We have linear topological isomorphisms
 \begin{align}
\label{prop5.eq1}(\cD' \widehat\otimes \cB_c)'_b & \cong \cD \widehat\otimes_\iota \cD'_{L^1} = \cD \widehat\otimes_\beta \cD'_{L^1} \\
& \cong (\cD' \widehat\otimes \dot{\cB})'_b \nonumber \\
\intertext{and linear isomorphisms}
  \nonumber (\cD' \widehat\otimes \dot\cB)' & \cong \cD ( \cD'_{L^1}; \e) = \cD(\cD'_{L^1}; \beta) \\
\nonumber  & = \Nu ( \cD', \cD'_{L^1}) \\
\nonumber  & \cong \cD'_{L^1, c} ( \cD; \e) = \cD'_{L^1, c} ( \cD; \beta).
 \end{align}
\end{proposition}
\begin{proof}
By Proposition \ref{prop5}$'$ we have $(\cD' \widehat\otimes \dot\cB)' = (\cD' \widehat\otimes \cB_c)'$ algebraically. Furthermore,
\[ (\cD' \widehat\otimes \dot\cB)'_b \cong (\cD' \widehat\otimes \cB_c)'_b, \]
because $\cD' \widehat\otimes \dot\cB$ is distinguished: the bounded sets of $\cD' \widehat\otimes \cB_c$ (and also of $\cD' \widehat\otimes \cB$) are contained in the weak closure of bounded sets in $\cD' \widehat\otimes \dot\cB$. Hence,
\[ (\cD' \widehat\otimes \cB_c)'_b \cong (\cD' \widehat\otimes \dot\cB)'_b. \]
Therefore, the first isomorphism in \eqref{prop5.eq1} results from Grothendieck's duality Corollary cited in the introduction \cite[Chap.~II, \S 4, n°1, Lemme 9, Corollaire, p.~90]{zbMATH03199982}: $\cD'$ is complete and nuclear, $\cB_c$ is complete \cite[p.~101]{FDVV} and semireflexive (see \cite[Proposition 1.3.1, p.~11]{zbMATH06308371} and \cite[p.~126]{zbMATH03145498}) and $(\cD' \widehat\otimes \dot\cB)'_b$ is complete, due to Proposition \ref{prop4}. The equality of the $\iota$- and the $\beta$-topology in the first line follows because $\cD$ and $\cD'_{L^1}$ are barrelled spaces \cite[p.~13]{zbMATH03145499}.
\end{proof}

\begin{remark*}
 We obtain, by Propositions \ref{prop5} and \ref{prop5}$'$, that $(\cD'(\dot\cB))'_b = \cD \widehat\otimes_\iota \cD'_{L^1}$ equals $\varinjlim_m ( \cD \widehat\otimes \cD^{\prime m}_{L^1})$ algebraically, which is a representation of the strong dual of $\cD'(\dot\cB)$ as a countable inductive limit.
\end{remark*}

In fact, for $K \in (\cD'(\dot\cB))'$ we conclude by the implication ``$\Rightarrow$'' above that there exist $m \in \bN_0$, $g_\alpha(x,y) \in \cD_x \widehat\otimes L^1_y$ with $K = \sum_{\abso{\alpha} \le m} \pd_y^\alpha g_\alpha$. Hence, $K \in \cD_x \widehat\otimes \cD^{\prime m}_{L^1, y}$. In order to see $\cD_x \widehat\otimes \cD^{\prime m}_{L^1,y} \subset \cD_x \widehat\otimes_\iota \cD'_{L^1,y}$ it suffices, due to ``$\Leftarrow$'' above, to show the implication
\[ K(x,y) \in \cD_x \widehat\otimes \cD^{\prime m}_{L^1} \Longrightarrow \exists g_\alpha (x,y) \in \cD_x \widehat\otimes L^1_{y} \textrm{ with }K = \sum_{\abso{\alpha} \le m } \pd_y^\alpha g_\alpha. \]
 However, this implication is a consequence of
\begin{gather*}
\cD^{\prime m}_{L^1} = \sum_{\abso{\alpha} \le m} \pd^\alpha L^1\quad\textrm{ and}\quad \cD_x \widehat\otimes \cD^{\prime m}_{L^1,y} = \sum_{\abso{\alpha} \le m}\pd_y^\alpha ( \cD_x \widehat\otimes L^1_{y}).
\end{gather*}

\begin{proposition}[Existence and uniqueness of the scalar product]\label{prop6}
 There is one and only one scalar product
\[
 \genfrac{}{}{0pt}{0}{{\langle\ ,\ \rangle}_x}{ {\langle\ ,\ \rangle}_y } \colon (\cD_x \widehat\otimes_\beta \cD'_{L^1,y} ) \times (\cD'_x \widehat\otimes \dot\cB_y) \to \bC
\]
which is partially continuous and coincides on $(\cD_x \otimes \cD'_{L^1,y}) \times (\cD'_x \otimes \dot\cB_y)$ with the product
$ {}_{\cD_x} \langle\ ,\ \rangle_{\cD'_x} \cdot {}_{\cD'_{L^1,y}} \langle\ ,\ \rangle_{\dot{\cB}_y}$.
\end{proposition}

\begin{proof}
 A \emph{first} proof follows from the ``Th\'eor\`emes de croisement'' \cite[Proposition 2, p.~18]{zbMATH03145499}.

A \emph{second} proof consists in the composition of the vectorial scalar product given by \cite[Proposition 10, p.~57]{zbMATH03145499}, i.e.,
\[ \langle\ ,\ \rangle_\iota \colon (\cD_x \widehat\otimes_\beta \cD'_{L^1,y} ) \times \cD'_x ( \dot\cB_y) \to \cD'_{L^1,y} \widehat\otimes_\iota \dot\cB_y = \cD'_{L^1,y} \widehat\otimes_\beta \dot\cB_y, \]
and the scalar product $\langle\ ,\ \rangle \colon \cD'_{L^1,y} \widehat\otimes_\iota \dot\cB_y \to \bC$.
\end{proof}

\begin{remark*}If $K(x,y) \in \cD_x \widehat\otimes_\iota \cD'_{L^1,y}$ has the representation
\[ K(x,y) = \sum_{\abso{\alpha} \le m} \pd_y^\alpha g_\alpha(x,y) \]
with $g_\alpha(x,y) \in \cD_x \widehat\otimes L^1_y$ and if $L(x,y) \in \cD'_x ( \dot\cB_y)$ then
 \[ \genfrac{}{}{0pt}{0}{{\langle\ ,\ \rangle}_x}{ {\langle\ ,\ \rangle}_y } \bigl( K(x,y), L(x,y) \bigr) = \sum_{\abso{\alpha} \le m} (-1)^{\abso{\alpha}} \int_{\bR^n} {}_{\cD'_x} \bigl \langle \pd_y^\alpha L(x,y), g_\alpha(x,y) \bigr\rangle_{\cD_x}\,\ud y.  \]
\end{remark*}

We find a third expression for the scalar product by means of vector-valued multiplication and integration:

\begingroup
\def\thetheorem{\ref{prop6}$'$}
\addtocounter{theorem}{-1}
\begin{proposition}[Scalar product]\label{prop6bis}
If $K(x,y) \in \cD_x \widehat\otimes_\beta \cD'_{L^1,y}$ and $L(x,y) \in \cD'_x \widehat\otimes \dot\cB_y$ then the scalar product $\genfrac{}{}{0pt}{1}{{\langle\ ,\ \rangle}_x}{ {\langle\ ,\ \rangle}_y }$ (Proposition \ref{prop6}) can also be expressed as
\[ \genfrac{}{}{0pt}{0}{{\langle\ ,\ \rangle}_x}{ {\langle\ ,\ \rangle}_y } \bigl( K(x,y), L(x,y)\bigr) = {}_{\cD'_{L^1,xy}} \bigl\langle K(x,y) \genfrac{}{}{0pt}{1}{\cdot_x}{\cdot_y} L(x,y), 1(x,y) \bigr\rangle_{\cB_{c, xy}}, \]
wherein $\genfrac{}{}{0pt}{1}{\cdot_x}{\cdot_y}$ denotes the vectorial multiplicative product
\[ (\cD_x \widehat\otimes_\beta \cD'_{L^1,y} ) \times (\cD'_x \widehat\otimes \dot\cB_y) \to \cE'_x \widehat\otimes_\e \cD'_{L^1,y}. \]
\end{proposition}
\endgroup

\begin{proof}
 The vectorial multiplicative product $\genfrac{}{}{0pt}{1}{\cdot_x}{\cdot_y}$ exists uniquely as the composition of the canonical mapping defined by the ``Th\'eor\`emes de croisement'' \cite[Proposition 2, p.~18]{zbMATH03145499},
\[ \can \colon (\cD_x \widehat\otimes_\beta \cD'_{L^1,y} ) \times (\cD'_x \widehat\otimes \dot\cB_y) \to (\cD_x \widehat\otimes_\beta \cD'_x) \bine ( \cD'_{L^1,y} \widehat\otimes_\beta \dot\cB_y) \]
and the $\bine$-product of the two multiplications
\[ \cD_x \widehat\otimes_\beta \cD'_x \xrightarrow{\cdot_x} \cE'_x\quad\textrm{and}\quad \cD'_{L^1,y} \widehat\otimes_\beta \dot\cB_y \xrightarrow{\cdot_y} \cD'_{L^1,y}, \]
namely
\[ K(x,y) \genfrac{}{}{0pt}{1}{\cdot_x}{\cdot_y} L(x,y) = [ ( \cdot_x) \bine (\cdot_y) ] \circ \can. \]
Note that this vectorial multiplication coincides with that defined in \cite[Proposition 32, p.~127]{zbMATH03145499}. Due to the uniqueness of the scalar product and the continuity of the embedding $\cE'_x \widehat\otimes \cD'_{L^1,y} \hookrightarrow \cD'_{L^1,xy}$ the result follows.
\end{proof}

\begin{proposition}[Existence of the regularization mapping and representation of its transpose]\label{prop7}
 The regularization mapping
\[ \dot\cB' \to \cD'_y \widehat\otimes \dot\cB_x,\quad S \mapsto S(x-y) \]
is well-defined, linear, injective and continuous. Its transpose
\[ \cD_y \widehat\otimes_\iota \cD'_{L^1,x} \to \cD_{L^1,c} \]
is linear, continuous and given by
\[ K(x,y) \mapsto {}_{\cB_{c,x}} \langle 1(x), K(x-y,x) \rangle_{\cD'_{L^1,x} ( \cD_{L^1,c,y})}. \]
\end{proposition}

\begin{proof}
1. The well-definedness is L.~Schwartz' classical regularization result \cite[Remarque 3°, p.~202]{TD}.

2. Due to the (sequentially) closed graph of the regularization mapping the continuity is implied by \cite[Chap.~I, Th\'eor\`eme B, p.~17]{zbMATH03199982}, if $\dot\cB'$ and $\cD' \widehat\otimes \dot\cB$ are ultrabornological. The sequence-space representation of $\dot\cB' \cong c_0 \widehat\otimes s'$ \cite[Theorem 3, p.~13]{bargetz} shows that $\dot\cB'$ is ultrabornological if \cite[Proposition 2, p.~75]{zbMATH03312960} is applied. The space $\cD' \widehat\otimes \dot\cB$ is ultrabornological by Proposition \ref{prop4}. Alternatively, for applying the closed graph theorem one can use that $\cD' \widehat\otimes \dot\cB$ has a completing web \cite[p.~736]{MR0306857}.

The transpose of the regularization mapping is continuous by \cite[Corollary to Proposition 3.12.3, p.~256]{zbMATH03230708}.

3. The representation in Proposition \ref{prop6bis} yields for $K(x,y) \in \cD_x \widehat\otimes \cD'_{L^1,y}$ and $S(y-x) \in \cD'_x \widehat\otimes \dot\cB_y$:
\[
\genfrac{}{}{0pt}{0}{\langle \ , \ \rangle_x}{\langle \ , \ \rangle_y} \bigl( K(x,y), S(y-x) \bigr) = {}_{\cD'_{L^1,xy}} \bigl\langle K(x,y) \genfrac{}{}{0pt}{1}{ \cdot_x }{ \cdot_y } S(y-x), 1(x,y) \bigr\rangle_{\cB_{c,xy}}. 
\]
The linear change of variables
\[
 \begin{aligned}
  x &= v-u \\
  y &= v
 \end{aligned}
\quad \begin{aligned}
         u &= y-x \\
v &= y
        \end{aligned}
\]
and the Theorem of Fubini \cite[Corollary, pp.~136, 137]{zbMATH03145498} imply that the last expression equals
\begin{multline*}
 {}_{\cB_{c, uv}} \langle 1(u,v), K(v-u, v) \genfrac{}{}{0pt}{1}{ \cdot_u }{ \cdot_v } ( S(u) \otimes 1(v)) \rangle_{\cD'_{L^1,uv}} \\
= {}_{\dot\cB'_u} \langle S(u), {}_{\cB_{c,v}} \langle 1(v), K(v-u, v) \rangle_{\cD'_{L^1, v} ( \cD_{L^1, c, u})} \rangle_{\cD_{L^1, c, u}} 
\end{multline*}
if we show that
\begin{equation}\label{eq4}
 K(v,u) \in \cD_v \widehat\otimes_\iota \cD'_{L^1, u} \Longrightarrow K(v-u, v) \in \cD'_{L^1, v} ( \cD_{L^1, c,u}). 
\end{equation}
Then, the multiplicative product $K(v-u, v) \genfrac{}{}{0pt}{1}{ \cdot_u }{ \cdot_v } ( S(u) \otimes 1(v))$ is defined as the image of $( K(v-u, v), S(u) \otimes 1(v) )$
under the mapping
\[ \cD'_{L^1, v} ( \cD_{L^1, c, u}) \times \cB_{c,v} ( \dot\cB'_u) \xrightarrow{ \genfrac{}{}{0pt}{1}{ \cdot_u }{ \cdot_v } } \cD'_{L^1, v} (\cD'_{L^1,c,u}). \]
It remains to prove the implication \eqref{eq4}: a vectorial regularization property similar to \cite[Proposition 15]{RCD} shows that
\begin{gather*}
 K(v,u) \in  \cD_v \widehat\otimes_\iota \cD'_{L^1,u} \subset \cD_v \widehat\otimes \cD'_{L^1,u} \\
\Longrightarrow K(v-w, u) \in \cD'_{L^1, u} \widehat\otimes_\e \cB_{c,v} \widehat\otimes_\e \cD_{L^1, c, w}
\end{gather*}
because the convolution with $\varphi \in \cD$ maps the space $\cD'_{L^1,c}$ continuously into $\cD_{L^1,c}$. The vector-valued multiplication $\cD'_{L^1} \times \cB_c (E) \xrightarrow{\cdot} \cD'_{L^1}(E)$
by \cite[Proposition 25, p.~120]{zbMATH03145499} then yields
\[ K(v-u,v) \in \cD_{L^1, v} \widehat\otimes_\e \cD_{L^1, c, u} = \cD'_{L^1, v} ( \cD_{L^1, c, u}). \qedhere\]
\end{proof}

\section{On the duals of tensor products \texorpdfstring{\\} -- two complements}\label{sec3}

The goal of this section is the formulation of propositions which yield, as special cases, the strong duals of the spaces $\cD' \widehat\otimes \cD'_{L^1}$ and $\cD' \widehat \otimes \dot\cB$. These spaces are the ``endpoints'' in the scale of reflexive spaces $\cD' \widehat\otimes \cD'_{L^p}$ and $\cD' \widehat\otimes \cD_{L^q}$, $1<p,q<\infty$, the duals of which can be determined by the Corollaire \cite[Chap.~II, \S 4, n°1, Lemme 9, Corollaire, p.~90]{zbMATH03199982} cited in the introduction.

\begin{proposition}[Dual of a completed tensor product]\label{prop8}
 Let $\cH = \varinjlim_k \cH_k$ be the strict inductive limit of nuclear Fr\'echet spaces $\cH_k$ and $F$ the strong dual of a distinguished Fr\'echet space. Then
 \[ (\cH'_b \widehat\otimes F)'_b \cong \overline{\cH}(F'_b) \coleq \varinjlim_k ( \cH_k(F'_b)). \]
 The space $\overline{\cH}(F'_b)$ is a complete, strict (LF)-space and $\cH'_b \widehat\otimes F$ is distinguished. If $F$ is reflexive, $\cH'_b \widehat\otimes F$ is reflexive, too.
\end{proposition}

\begin{proof}
 By \cite[Prop.~22, p.~103]{zbMATH03145499} we have algebraically
 \[ (\cH'_b \widehat\otimes F)' \cong F'_c ( \cH; \e ) = F_c' ( \cH; \beta) = \varinjlim_k ( F_c'(\cH_k; \beta)) \]
 due to the reflexivity of $\cH$ and due to the fact that a linear and continuous map $T \colon F \to \cH$ is bounded if and only if there exists $k$ and a $0$-neighborhood $U$ in $F$ such that $T$ maps into $\cH_k$ and $T(U) \subseteq \cH_k$ is bounded. Because $F$ is the strong dual of a distinguished Fr\'echet space,
 \[ F_c' ( \cH_k; \beta) = F_c' ( \cH_k) \]
 by \cite[p.~98, b)]{zbMATH03145499}. Hence,
 \[ F_c'(\cH_k) = \cL_\e ( (\cH_k)'_c, F_c') \cong \cL_c ( F, \cH_k) = F_b' \widehat\otimes \cH_k \]
 (see \cite[Chap.~I, p.~75]{zbMATH03199982}). All together,
 \[ (\cH_b' \widehat\otimes F)' = \varinjlim_k \cH_k(F_b') = \overline{\cH}(F_b'). \]
 
The strong dual topology on $(\cH'_b (F))'$ is finer than the topology of uniform convergence on products of bounded subsets of $\cH'_b$ and $F$ \cite[Prop.~22, p.~103]{zbMATH03145499}, i.e., the embedding $(\cH'_b (F))'_b \hookrightarrow \cH \widehat\otimes F_b'$ is continuous.
 
In order to show the continuity of the map
 \[ \Id \colon \overline{\cH}(F_b') \to (\cH'_b (F))'_b \]
 we use an idea from the proof of \cite[Prop.~3, p.~542]{zbMATH03163214}: it suffices that bounded sets in $\overline{\cH}(F_b')$ are bounded in $(\cH'_b(F))'_b$ because $\overline{\cH}(F_b')$ is bornological (note that $F'_b$ is a Fr\'echet space by \cite[p.~64]{zbMATH03092845}, and that the inductive limit of the Fr\'echet spaces $\cH_k \widehat\otimes F'_b = \cH_k ( F_b')$ is bornological). If $H \subseteq \overline{\cH}(F_b')$ is bounded then by the regularity of the inductive limit there exists $k$ such that $H$ is bounded in $\cH_k(F_b')$. By \cite[Chap.~II, \S 3, n°1, Prop.~12, p.~73]{zbMATH03199982} there are bounded subsets $A \subseteq \cH_k$ and $B \subseteq F'_b$ such that $H$ is contained in the closed absolutely convex hull of $A \otimes B$. For each $T \in \cH'_b \widehat\otimes F \cong \cL_b(\cH, F)$ the set $T(H)$ is contained in the closed absolutely convex hull of $T(A \otimes B) = \langle T(A), B \rangle$, which is bounded because $T(A) \subseteq F$ is bounded and $B$ is equicontinuous. Hence, we see that $H$ is weakly bounded and thus bounded in $(\cH_b'(F))'_b$ due to the completeness of $\cH'_b(F)$.
 \end{proof}

\begin{remark*}
 \begin{enumerate}[label=(\arabic*)]
  \item The strong topology of $(\cH'_b(F))'_b$ coincides with the topology induced by $\cL_b(\cH'_b, F'_b) = \cH \widehat\otimes F_b'$. By \cite[Prop.~22, p.~103]{zbMATH03145499} this is equivalent with saying that bounded subsets of $\cH'_b(F) = \cH'_b \widehat\otimes F$ are $\beta$-$\beta$-decomposable.
  \item If the space $F$ is the strong dual of a \emph{reflexive} Fr\'echet space then $\cH'_b \widehat\otimes F$ is reflexive too, i.e.,
  \[ (\cH'_b \widehat\otimes F)'_b \cong \overline{\cH}(F'_b) \]
  and
  \[ (\overline{\cH}(F'_b))'_b \cong \cH'_b \widehat\otimes F \]
 (for the second part use a suitable generalization of \cite[Corollaire 3, p.~104]{zbMATH03145499}). This assertion generalizes the reasoning in \cite[p.~314, 4.2, p.~315]{zbMATH03223921}.
 \item The following list of distribution spaces illustrates possibilities of the applicability of Proposition 8:
 \begin{alignat*}{2}
  \cH &: \cD, \cD_+, \cD_-, \cD_{+\Gamma}\quad&&\textrm{(cf. \cite[p.~154, p.~186]{zbMATH03145499}}),\\
  F & : \cD'_{L^1}, \cD'^m_{L^1}, \ell^1, \cS'^m, \cE'^m \quad &&(m \in \bN_0).
 \end{alignat*}
One has to show that $\cS^m$ and $\cE^m$ are distinguished.
 \item The dual of the space of partially summable distributions $\cD' \widehat\otimes \cD'_{L^1}$ was given first in \cite[Prop.~3(2), p.~541]{zbMATH03163214}, i.e., $(\cD' \widehat\otimes \cD'_{L^1})'_b = \overline{\cD}(\cB) = \cD \widehat\otimes_\iota \cB$.
 \end{enumerate}
\end{remark*}

By considering the proof of Proposition \ref{prop5}, i.e., $(\cD'(\dot\cB))'_b = (\cD' \widehat\otimes \cB_c)'_b = \cD \widehat\otimes \cD'_{L^1}$, we were led to the following modification of A.~Grothendieck's Corollary \cite[Chap.~II, \S 4, n°1, Lemme 9, Corollaire, p.~90]{zbMATH03199982} cited in the introduction.


 \begin{proposition}[Dual of a completed tensor product]\label{prop9}
  Let $\cH$ be a Hausdorff, quasicomplete, nuclear, locally convex space with the strict approximation property, $F$ a quasicomplete, semireflexive, locally convex space. Let $F_0$ be a locally convex space such that
  \[ (\cH \widehat\otimes F)'_b = (\cH \widehat \otimes F_0)'_b \]
  and $(\cH \widehat\otimes F_0)'_b$ is complete. Then,
  \[ (\cH \widehat\otimes F)'_b \cong \cH'_b \widehat\otimes_\iota F'_b \]
  and $(\cH \widehat\otimes F)'_b$ is semireflexive.
 \end{proposition}
The \emph{proof} is an immediate consequence of Proposition \ref{prop10}. The semireflexivity is a consequence of \cite[Corollaire 2, p.~118]{zbMATH03199982}.

\begin{remark*}
 \begin{enumerate}[label=(\arabic*)]
  \item A.\ Grothendieck's hypotheses ``\,$\cH$ complete, $F$ complete and semireflexive'' are weakened by the assumption of quasicompleteness at the expense of the additional hypothesis of the strict approximation property for $\cH$. The completeness of the strong dual $(\cH \widehat\otimes F)'_b$ is implied by the existence of an additional space $F_0$ with the corresponding property.
  \item By checking the hypotheses of Proposition \ref{prop9} we have shown in Proposition \ref{prop5} that
  \[ (\cD' \widehat\otimes \dot\cB)'_b \cong \cD \widehat\otimes_\iota \cD'_{L^1}. \]
  Two other applications to concrete distribution spaces are:
  \[ (\cD' \widehat\otimes \dot\cB')'_b \cong \cD \widehat\otimes_\iota \cD_{L^1} = \overline{\cD}(\cD_{L^1}) \]
  with $F_0 = ( \cD'_{L^\infty}, \kappa ( \cD'_{L^\infty}, \cD_{L^1}))$, and
  \[ (\cD' \widehat\otimes c_0)'_b \cong \cD \widehat\otimes_\iota \ell^1 = \overline{\cD}(\ell^1) \]
  with $F_0 = (\ell^\infty, \kappa(\ell^\infty, \ell^1)$.
\item As an application of Proposition \ref{prop9} we see that spaces like $\cS' \widehat\otimes \dot\cB$ or $\cO_M \widehat\otimes c_0$ are distinguished. This does not follow from \cite[Chap.~II, \S3, n°2, Corollaire 2, p.~77]{zbMATH03199982}. In fact, $(\cS' \widehat\otimes \dot\cB)'_b \cong \cS \widehat\otimes_\iota \cD'_{L^1}$ and $\cS \widehat\otimes_\iota \cD'_{L^1}$ is barrelled by \cite[Chap.~I, p.~78]{zbMATH03199982}. 
 \end{enumerate}
\end{remark*}

The proof of Proposition \ref{prop9} rests on a generalization of Grothendieck's Corollary on duality (cf.~\cite[Chap.~II, \S 4, n°1, Lemme 9, Corollaire, p.~90]{zbMATH03199982}) which we prove now.

\begin{proposition}[Duals of tensor products]\label{prop10}
\leavevmode\\
\underline{Hypothesis 1:} Let $E$ be a nuclear, $F$ a locally convex space.

 Then:
\begin{enumerate}[label=(\roman*)]
 \item\label{prop9.1} Every element $u$ of the dual $\cB(E,F) = (E \widehat\otimes_\pi F)' = (E \otimes_\pi F)'$ is the image (under a canonical mapping) of an element $u_0$ of a space $E'_A \widehat\otimes_\pi E'_B$, where $A$ and $B$ are absolutely convex, weakly closed, equicontinuous subsets of $E'$ and $F'$, respectively.
\item\label{prop9.2} If the element $u_\iota$ of $E' \wideparen \otimes_\iota F' \subset E' \widehat \otimes_\iota F'$ defined by $u_0$ is zero then $u$ is zero, from which we have a canonical injection of $\cB(E,F)$ into $E' \wideparen \otimes_\iota F' \subset E' \widehat\otimes_\iota F'$.
\item\label{prop9.3} $\cB(E,F)$ is dense in $E' \widehat\otimes_\iota F'$ and strictly dense in $E' \wideparen\otimes_\iota F'$.
\end{enumerate}
If, in addition, we have

\hangindent=3em
\hangafter=1
\underline{Hypothesis 2:} $E$ is quasicomplete and has the strict approximation property and $F$ is quasicomplete,

then we obtain:
\begin{enumerate}[label=(\roman*),resume]
 \item\label{prop9.4} The topology $t_\iota$ induced from $E' \wideparen \otimes_\iota F'$ on $\cB(E,F)$ is finer than the topology $t_b$ of $(E \wideparen\otimes_\pi F)'_b$, i.e., $t_\iota \ge t_b$.
\end{enumerate}
If, in addition, we have

\underline{Hypothesis 3:} $F$ is semireflexive,

then we obtain:
\begin{enumerate}[label=(\roman*),resume]
 \item\label{prop9.5} The topology induced from $E' \widehat\otimes_\iota F'$ on $\cB(E,F)$ is identical with the topology of the strong dual $(E \wideparen\otimes_\pi F)'_b$, i.e., $t_\iota = t_b$. $E' \widehat\otimes_\iota F'$ is the completion of $\cB(E,F)$.
 \item\label{prop9.6} $(E \wideparen\otimes_\pi F)'_b$ $\cong E' \widehat\otimes_\iota F'$ if and only if $(E \wideparen\otimes_\pi F)'_b$ is complete.
 
 $(E \wideparen\otimes_\pi F)'_b \cong E' \wideparen\otimes_\iota F'$ if and only if $(E \wideparen\otimes_\pi F)'_b$ is quasi-complete.
 \item\label{prop9.7} $E \widehat\otimes F$ is semireflexive.
\end{enumerate}
\end{proposition}
\begin{proof}
We shall modify the proof of Grothendieck and give more details.

\ref{prop9.1} If $u \in \cB(E,F)$ then the nuclearity of $E$ implies the existence of zero-neighborhoods $U$ in $E$ and $V$ in $F$ and of sequences $(e_n')$ in $E'_{U^\circ}$ and $(f_n')$ in $F'_{V^\circ}$ such that
\begin{equation}\label{prop9.eq1}
 \sum_{n=1}^\infty \norm{e_n'}_{U^\circ} \norm{f_n'}_{V^\circ} < \infty
\end{equation}
and
\begin{equation}\label{prop9.eq2}
 u(e,f) = \sum_{n=1}^\infty \langle e, e_n' \rangle \langle f, f_n' \rangle \qquad \forall (e,f) \in E \times F
\end{equation}
by \cite[Chap.~II, \S 2, n°1, Corollaire 4 to Th\'eor\`eme 6, p.~39]{zbMATH03199982} or \cite[21.3.5, p.~487]{Jarchow}. Setting $A \coleq U^\circ$, $B \coleq V^\circ$ we then define $u_0 \in E_A' \widehat\otimes_\pi E_B'$ by
\[ u_0 \coleq \sum_{n=1}^\infty e_n' \otimes f_n'. \]
The series in \eqref{prop9.eq2} converges because
\[ \sum_{n=1}^\infty \abso{ \langle e, e_n' \rangle \langle f, f_n' \rangle } \le \left( \sum_{n=1}^\infty \norm{ e_n' }_{U^\circ} \norm{f_n'}_{V^\circ} \right) \cdot \norm{e}_U \cdot \norm{f}_V \]
due to Lemma \ref{prop9.lem1} and inequality \eqref{prop9.eq1}. Moreover, $u_0 = \sum e_n' \otimes f_n'$ converges in the Banach space $E_A' \widehat\otimes_\pi E_B'$:
\begin{align*}
 \norm{u_0} &= \norm{ \sum_{n=1}^\infty ( e_n' \otimes f_n') } = \inf \left\{\, \sum_{n=1}^\infty \norm{a_n}_{U^\circ} \norm{b_n}_{V^\circ}\ :\  u_0 = \sum_{n=1}^\infty a_n \otimes b_n \,\right\} \\
& \le \sum_{n=1}^\infty \norm{e_n'}_{U^\circ} \norm{f_n'}_{V^\circ}.
\end{align*}
Next let us describe in detail \emph{the canonical mapping} $E_A' \widehat\otimes_\pi F_B' \to \cB(E,F)$ as the composition of the following three mappings
\begin{equation}\label{prop9.eq3}
 E_A' \widehat\otimes_\pi F_B' \to (E_U)' \widehat\otimes_\pi (F_V)' \to (E_U \widehat\otimes_\pi F_V)' \to (E \widehat\otimes_\pi F)' = \cB(E,F).
\end{equation}
For \emph{the first mapping in \eqref{prop9.eq3}} we use that ${}^t \Phi_U \colon (E_U)' \to E_A'$ is an isomorphism with inverse $({}^t \Phi_U)^{-1} \colon E_A' \to (E_U)'$.

The \emph{second mapping in \eqref{prop9.eq3}} is the continuous extension of the linear map on $(E_U)' \otimes (E_V)'$ corresponding to the continuous bilinear map
\begin{align*}
 (E_U)' \times (E_V)' &\to (E_U \widehat\otimes_\pi F_V)', \\
(e', f') &\mapsto e' \otimes f'.
\end{align*}

The \emph{third mapping in \eqref{prop9.eq3}} is given as the transpose of $\Phi_U \otimes \Phi_V \colon E \widehat\otimes_\pi F \to E_U \widehat\otimes_\pi F_V$.

\emph{The image of $u_0$ in $\cB(E,F)$ coincides with $u$:} denoting the image of $u_0$ in all spaces appearing in \eqref{prop9.eq3} by $u_0$ we obtain by going from right to left in the composition above:
\begin{align*}
 u_0(e,f) & = u_0 ( e \otimes f) = u_0 ( \Phi_U(e) \otimes \Phi_V(f)) \\
&= \left( \sum_{n=1}^\infty ( {}^t \Phi_U)^{-1} e_n' \otimes ({}^t \Phi_V)^{-1} f_n'\right)\bigl(\Phi_U(e) \otimes \Phi_V(f)\bigr) \\
&= \sum_{n=1}^\infty \left\langle \Phi_U(e), ( {}^t \Phi_U)^{-1} e_n' \right\rangle \left\langle \Phi_V(f), ({}^t \Phi_V)^{-1} f_n' \right\rangle \\
&= \sum_{n=1}^\infty \langle e, e_n' \rangle \langle f, f_n' \rangle = u(e,f).
\end{align*}

\ref{prop9.2} We also have a canonical mapping
\[ ({}^t \Phi_U) \otimes ( {}^t \Phi_V ) \colon E_A' \widehat\otimes_{\pi} F_B' = E_A' \wideparen\otimes_\iota F_B' \to E' \wideparen \otimes_\iota F' \subset E' \widehat\otimes_\iota F' \]
which maps $u_0$ to an element $u_\iota \in E' \wideparen\otimes_\iota F'$. Due to \[ \left( ({}^t \Phi_U) \otimes ({}^t \Phi_V)\right)(u_0) = u_0 \circ (\Phi_U \otimes \Phi_V)\quad\textrm{in }\cB(E,F) \]
we conclude from $u_\iota = 0$ that $u_0$ vanishes on $\widehat E_U \times \widehat F_V$. By \cite[Chap.~I, \S 5, n°2, Corollaire 1, p.~181]{zbMATH03199982} or \cite[Theorem 18.3.4, p.~406]{Jarchow} the canonical mapping
\begin{align*}
 E_A' \widehat\otimes F_B' &\to \cB ( \widehat E_U, \widehat F_V)
\end{align*}
is injective because the zero-neighborhood $U$ can be chosen in such a manner that $\widehat E_U$ is a Hilbert space \cite[Chap.~II, \S 2, n°1, Lemme 3, p.~37]{zbMATH03199982}, and, therefore, $\widehat E_U$ is reflexive and has the approximation property. Hence, the vanishing of $u_0$ on $\widehat E_U \times \widehat F_V$ implies $u_0=0$ and a fortiori $u=0$.

\ref{prop9.3} follows from $E' \otimes F' \subset \cB(E,F) \subset E' \wideparen \otimes_\iota F' \subset E' \widehat\otimes_\iota F'$.

\ref{prop9.4} If, in addition, Hypothesis 2 is fulfilled, we obtain by Lemma \ref{prop9.lem2} below that $E \wideparen\otimes_\pi F \cong \cB_\e^h (E_c', F_c')$. Two topologies on $\cB(E,F)$ can be defined corresponding to the following two dual systems (cf.~\cite[Chap.~III, \S 2, p.~46]{MR0350361} or \cite[Chap.~3, \S 2, p.~183]{zbMATH03230708}):
\[
 \xymatrix@C=1ex{
\bigl ( \cB(E,F) \ar[d]_j , & E \wideparen\otimes_\pi F \ar@{}[r]|-*[@]\txt{$\cong$} & \cB_\e^h ( E'_c, F'_c ) \ar[d]_m \bigr ) \\
\bigl ( E' \wideparen\otimes_\iota F',  & (E' \wideparen\otimes_\iota F')'_\e \ar@{}[r]|-*[@]\txt{$\cong$} & \cB_\e^s ( E', F') \bigr ).
}
\]
The mapping $j$ is defined as the injection
\[ \cB(E,F) \to E' \wideparen \otimes_\iota F',\quad u \mapsto \sum_{n=1}^\infty e_n' \otimes f_n' \]
investigated in \ref{prop9.1} and \ref{prop9.2}.

The mapping $m$ is the injection of the space $\cB^h(E_c', F_c')$ of hypocontinuous bilinear forms on $E_c' \times F_c'$ into the space $\cB^s(E', F')$ of separately continuous bilinear forms on $E' \times F'$. Furthermore, we have canonical isomorphisms
\begin{align*}
 k \colon E \wideparen \otimes_\pi F &\to \cB^h_\e ( E_c', F_c'),\quad z \mapsto [ (e', f') \mapsto \langle e' \otimes f', z \rangle ],\\
\ell \colon \cB^s ( E', F') & \to (E' \widehat\otimes_\iota F')',\quad \ell(w)(e' \otimes f') = w(e', f').
\end{align*}
On the one hand, we consider on $\cB(E,F)$ the relative topology $t_\iota$ with respect to the embedding $j$. Because the topology of $E' \wideparen \otimes_\iota F'$ is the topology of uniform convergence on equicontinuous subsets of $(E' \wideparen\otimes_\iota F')' = \cB^s(E', F')$ and equicontinuous subsets of this dual space are precisely the equicontinuous subsets of $\cB^s(E', F')$ \cite[Chap.~I, \S 3, n°1, Proposition 13, p.~73]{zbMATH03199982} we see that a zero-neighborhoodbase of $t_\iota$ is given by the sets $j^{-1}(\ell(B')^\circ)$ with separately equicontinuous subsets $B'$ of $\cB^s(E', F')$.

Let us now show that $t_\iota$ is finer than $t_b$; in particular, that for a given bounded set $B$ in $E \wideparen\otimes_\pi F$ the set $B' \coleq m(k(B))$ is separately equicontinuous and that $j^{-1}(\ell(B')^\circ) = B^\circ$:

If $B \subset E \wideparen\otimes_\pi F \cong \cB^h_\e ( E_c', F_c' )$ is bounded the set $B(.,f')$, for fixed $f' \in F'$, is bounded on all equicontinuous sets $U^\circ$, $U$ an absolutely convex, closed zero-neighborhood in $E$, i.e., $\forall U$ $\exists \lambda_U > 0$ such that
\[ \sup_{e' \in U^\circ} \abso { B(e', f') } \le \lambda_U, \]
so $B(., f') \subset \lambda_U U^{\circ\circ} = \lambda_U U$ due to $(E'_c)' = E$. Hence, $B(., f')$ is bounded in $E$ and $B(.,f')^\circ$ is a zero-neighborhood in $E'_b$, which means that $B(.,f')$ is equicontinuous on $E'_b$. Similarly, $B(e', .)$ is equicontinuous on $F'_b$ for $e' \in E'$ fixed. Therefore, $B' \coleq m(k(B))$ is a separately equicontinuous subset of $\cB^s(E', F')$ which implies that $(\ell(B'))^\circ$ is a zero-neighborhood in $t_\iota$, which together with $j^{-1}(\ell(B')^\circ) = B^\circ$ will imply $t_b \le t_\iota$.

In order to prove $j^{-1}(\ell(B')^\circ)=B^\circ$ for a given bounded subset $B \subset E \wideparen \otimes_\pi F$ with $B' = m(k(B))$ we write down the involved mappings explicitly. Let $u \in \cB(E,F)$ with $j(u) = \sum_{i=1}^\infty e_n' \otimes f_n'$ and $z \in E \wideparen\otimes_\pi F$. Then,

\newcommand*\circled[1]{\tikz[baseline=(char.base)]{\node[shape=circle,draw,inner sep=1pt] (char) {#1};}}
\newcommand*\smallcircled[1]{\tikz[baseline=(char.base)]{\node[shape=circle,draw,inner sep=1pt] (char) {\tiny #1};}}

\begin{align*}
 \langle j(u), \ell(m(k(z)))\rangle & \stackrel{\smallcircled{1}}{=} \ell ( m ( k(z)))\left(\sum_{n=1}^\infty e_n' \otimes f_n'\right) \\
 & \stackrel{\smallcircled{2}}{=} \sum_{n=1}^\infty \ell ( m ( k(z)))(e_n' \otimes f_n') \stackrel{\smallcircled{3}}{=} \sum_{n=1}^\infty m ( k(z))(e_n', f_n') \\
 & \stackrel{\smallcircled{4}}{=} \sum_{n=1}^\infty k(z)(e_n', f_n') \stackrel{\smallcircled{5}}{=} \sum_{n=1}^\infty (e_n' \otimes f_n')(z) \stackrel{\smallcircled{6}}{=} \langle u, z \rangle
\end{align*}

\circled{1} is the definition of $j$, \circled{2} follows from the continuity of $\ell(m(k(z)))$. \circled{3}, \circled{4} and \circled{5} are the definitions of $\ell$, $m$ and $k$, respectively. The equality \circled{6} is a consequence of the equality for $z$ in the strictly dense subspace $E \otimes F$ of $E \wideparen\otimes F$ and the continuity of
\[ z \mapsto \sum_{n=1}^\infty ( e_n' \otimes f_n')(z) \]
which follows from the representation
\[ \sum_{n=1}^\infty (e_n' \otimes f_n')(z) = \sum_{n=1}^\infty \norm{e_n'}_{U^\circ} \norm{f_n'}_{V^\circ} \left( \frac{e_n'}{\norm{e_n'}_{U^\circ}} \otimes \frac{f_n'}{\norm{f_n'}_{V^\circ}} \right)(z) \]
and the inequality (using Lemma \ref{prop9.lem1})
\begin{align*}
\abso{\sum_{n=1}^\infty ( e_n' \otimes f_n')(z)} & \le \sum_{n=1}^\infty \norm{e_n'}_{U^\circ} \norm{f_n'}_{V^\circ} \abso{ \left( \frac{e_n'}{\norm{e_n'}_{U^\circ}} \otimes \frac{f_n'}{\norm{f_n'}_{V^\circ}}\right)(z)} \\
& \le \sum_{n=1}^\infty \norm{e_n'}_{U^\circ} \norm{f_n'}_{V^\circ} \norm{ \frac{e_n'}{\norm{e_n'}_{U^\circ}} \otimes \frac{f_n'}{\norm{f_n'}_{V^\circ}} }_{U^\circ \otimes V^\circ} \norm{z}_{U \otimes V} \\
& \le \left( \sum_{n=1}^\infty \norm{e_n'}_{U^\circ} \norm{f_n'}_{V^\circ} \right) \norm{z}_{U \otimes V}.
\end{align*}
Note that the sets $\{ e_n' / \norm{e_n'}_{U^\circ} \}$ and $\{ f_n' / \norm{f_n'}_{V^\circ} \}$ are equicontinuous and hence also their tensor product is equicontinuous by \cite[p.~14]{zbMATH03145499}.

\ref{prop9.5} The space $E$ is nuclear and quasicomplete and, hence, $E$ is semireflexive (see \cite[Prop.~3, exp.~17, p.~5]{semschwartz}, \cite[Chap.~II, Cor.~1, p.~38]{zbMATH03199982} and \cite[Cor.~to Prop.~3.15.4, p.~277 and Prop.~3.9.1, p.~231]{zbMATH03230708}). Thus, with Hypothesis 3 both spaces $E$ and $F$ are semireflexive. Moreover, by \cite[Chap.~I, p.~11]{zbMATH03199982}:
\begin{align*}
 \cB^s ( E_b', F_b' ) &= \cL(E_b', (F_b')'_\sigma) = \cL ( E_b', (F_\sigma')'_\sigma) \\
 &= \cB^s ( E_b', F_\sigma') = \dotsc = \cB^s ( E_\sigma', F_\sigma'),\\
 \cB^s(E_c', F_c') &= \cB^s ( E_\sigma', F'_\sigma) \textrm{ and }\\
 \cB^h (E_b', F_b') &= \cB^s(E_b', F_b'),
\end{align*}
because $E'_b$ and $F'_b$ are barrelled \cite[Prop.~3.8.4, p.~228]{zbMATH03230708}. Hence,
\[ \cB_\e^h ( E_c', F_c') = \cB_\e^h ( E_b', F_b') = \cB^s_\e (E_\sigma', F_\sigma').\]
Let us show that $t_b$ is finer than $t_\iota$, $t_b \ge t_\iota$: if $B \subset \cB^s(E', F')$ is separately equicontinuous, $B$ is pointwise bounded such that Mackey's theorem \cite[IV.1, Prop.~1 (ii)]{zbMATH03757085} implies its boundedness in $\cB^s_\e( E_\sigma', F_\sigma')$, i.e., $B$ is bounded in $E \wideparen\otimes_\pi F$. In virtue of \ref{prop9.3}, $E' \widehat\otimes_\iota F'$ is the completion of $\cB(E,F)$ and $E' \wideparen \otimes_\iota F'$ its quasi-completion.

\ref{prop9.6} follows from $(E \wideparen\otimes_\pi F)' = \cB(E,F)$.

\ref{prop9.7} is a consequence of \cite[Chap.~I, \S 4, n°2, Corollaire of Prop.~24, p.~118]{zbMATH03199982}.
\end{proof}

\begin{lemma}\label{prop9.lem1}
 Let $U$ be an absolutely convex zero-neighborhood in $E$, $e' \in E'_{U^\circ}$ and $e \in E$. Then $\abso{ \langle e,e' \rangle} \le \norm{e}_U \norm{e'}_{U^\circ}$.
\end{lemma}
\begin{proof}
 The inequality follows either by restricting the scalar product on $E \times E'$ to the Banach spaces $E_U \times E'_{U^\circ}$ (taking into account that $(E_U)' = E'_{U^\circ}$) or, more elementary:
 \begin{align*}
  \norm{e'}_{U^\circ} = \sup_{e \in U}\abso{ \langle e, e' \rangle} & \Longrightarrow \forall e \in U: \abso{ \langle e, e'\rangle} \le \norm{e'}_{U^\circ} \\
  & \Longrightarrow \forall e \in E, \forall \e > 0: \abso{\langle \frac{e}{\norm{e}_U + \e}, e' \rangle } \le \norm{e'}_{U^\circ},
 \end{align*}
i.e., $\abso{\langle e,e' \rangle} \le \norm{e}_U \norm{e'}_{U^\circ}$.
\end{proof}

\begin{lemma}\label{prop9.lem2}
 By assuming the hypotheses 1 and 2 on $E$ and $F$ we have $E \wideparen \otimes_\pi F \cong \cB_\e^h(E_c', F_c')$.
\end{lemma}
\begin{proof}
 The nuclearity of $E$ implies that $E \wideparen\otimes_\pi F = E \wideparen\otimes_\e F$. By the quasi-completeness of $E$ and $F$ and the strict approximation property of $E$, we see that $E \wideparen\otimes_\e F = E \bine F = \cB^h_\e(E_c', F_c')$ \cite[Corollaire 1, p.~47]{zbMATH03145498}.
\end{proof}

\section{Decomposition of the space \texorpdfstring{$\dot\cB_{xy}'$}{Ḃ\textunderscore xy}}\label{sec4}

As mentioned in the introduction the decomposition
\[ \dot\cB_{xy} = \dot\cB \widehat\otimes_\e \dot\cB_y \]
 is proven in \cite[Proposition 17, p.~59]{zbMATH03145498}. It is an analogue of A.~Grothendieck's example
\[ \cC_0 (M \times N) = \cC_0(M) \widehat\otimes_\e \cC_0(N) \]
for locally compact topological spaces $M$ and $N$ \cite[Chap.~I, p.~90]{zbMATH03199982}. Thus, it seems to us of its own interest to state an analogue decomposition of $\dot\cB'_{xy}$, besides its applicability in proving the equivalence \eqref{eq1} $\Longleftrightarrow$ \eqref{eq3} hinted at in the introduction.

\begin{proposition}[Decomposition of $\dot\cB'_{xy}$]\label{prop11}
\[ \dot\cB'_{xy} = \dot\cB'_x \widehat\otimes_\e \dot\cB'_y = \dot\cB'_x \bine \dot\cB'_y = \dot\cB_x'(\dot\cB'_y). \]
\end{proposition}
\begin{proof}
First, we see that the isomorphism $\dot\cB'_{xy} \cong \dot\cB'_x \widehat\otimes_\e \dot\cB'_y$ is a consequence of the sequence-space representation $\dot\cB' \cong c_0 \widehat\otimes s'$ \cite[Theorem 3, p.~13]{bargetz}, A.~Grothendieck's example above and the commutativity of the $\bine$-product:
\[
 \dot\cB'_{xy} \cong c_{0,jm} \widehat\otimes s'_{kl} \cong (c_{0,j} \widehat\otimes_\e c_{0,m}) \widehat\otimes (s'_k \widehat\otimes s'_l) \cong (c_{0,j} \widehat\otimes s'_k) \widehat\otimes_\e ( c_{0,m} \widehat\otimes s'_l) \cong \dot\cB'_x \widehat\otimes_\e \dot\cB'_y.
\]

Alternatively, for the algebraic equality $\dot\cB'_{xy} = \dot\cB'_x \widehat\otimes_\e \dot\cB'_y$, the characterization of $\dot\cB'$ by regularization yields for $K(x,y) \in \cD'_{xy}$:
\begin{align*}
 K(x,y) \in \dot\cB'_{xy} & \Longleftrightarrow K(x-z, y-w) \in \cD'_{zw} \widehat\otimes \dot\cB_{xy}
\intertext{and, hence, by \cite[Prop.~17, p.~59 and Prop.~28, p.~98]{zbMATH03145498}}
& \Longleftrightarrow K(x-z, y-w) \in (\cD'_z \widehat\otimes \cD'_w) \widehat\otimes (\dot\cB_x \widehat\otimes_\e \dot\cB_y).
\intertext{By the commutativity of the $\bine$-product we obtain}
& \Longleftrightarrow K(x-z, y-w) \in \cD'_z ( \dot\cB_x ( \cD'_w \widehat\otimes \dot\cB_y)),
\intertext{and by the vectorial regularization property \cite[Proposition 15, p.~11]{RCD}}
& \Longleftrightarrow K(x, y-w) \in \dot\cB'_x ( \cD'_w \widehat\otimes \dot\cB_y) = (\cD'_w \widehat\otimes \dot\cB_y)(\dot\cB'_x) \\
& \Longleftrightarrow K(x,y) \in \dot\cB'_y \widehat\otimes_\e \dot\cB'_x = \dot\cB'_y ( \dot\cB'_x) = \dot\cB'_x ( \dot\cB'_y).
\end{align*}

For the topological equality we show that any 0-neighborhood in $\dot\cB'_x \bine \dot\cB'_y = \cL_\e ( \cD_{L^1, c, x}, \dot\cB'_y)$ is a neighborhood in $\dot\cB'_{xy}$, i.e., the topology of $\dot\cB'_{xy}$ is finer than the topology on $\dot\cB'_{xy}$ induced by $\dot\cB'_x \bine \dot\cB'_y$, i.e., $\Id \colon \dot\cB'_{xy} \to \dot\cB'_x \widehat\otimes_\e \dot\cB'_y$
is continuous: a base of 0-neighborhoods in $\cL_\e ( \cD_{L^1, c, x}, \dot\cB'_x)$ is given by means of bounded subsets $B_1 \subset \cD_{L^1,x}, B_2 \subset \cD_{L^1,y}$ in the form of sets
\begin{align*}
W(B_1, B_2^\circ) & = \{\, K(x,y) \in \dot\cB'_x \bine \dot\cB'_y : K(B_1) \subset B_2^\circ\,\} \\
& = \{\, K \in \dot\cB'_{xy} : \abso{K(B_1, B_2)} \le 1\,\}
\end{align*}
But the set $W(B_1, B_2^\circ)$ is a 0-neighborhood in $\dot\cB_{xy}$ because $B_1 \otimes B_2$ is a bounded set in $\cD_{L^1,xy}$. Hereby we made use of $(\dot\cB')' \cong \cD_{L^1}$.

Due to the isomorphism
\[ \dot\cB'_x \widehat\otimes_\e \dot\cB'_y \cong c_{0,jm} \widehat\otimes s'_{kl} \]
the space $\dot\cB'_x \widehat\otimes_\e \dot\cB'_y$ is ultrabornological (apply \cite[Proposition 2, p.~75]{zbMATH03312960}). As seen in the proof (2.) of Proposition \ref{prop7}, the space $\dot\cB'_{xy}$ is also ultrabornological. Hence, \cite[Chap.~I, Th\'eor\`eme B, p.~17]{zbMATH03199982} implies that $\Id$ is an isomorphism.
\end{proof}

{\bfseries Acknowledgments. } E.~A.~Nigsch was supported by the Austrian Science Fund (FWF) grants P23714 and P26859.

\newcommand{\bibarxiv}[1]{arXiv: \href{http://arxiv.org/abs/#1}{\texttt{#1}}}
\newcommand{\bibdoi}[1]{{\sc doi:} \href{http://dx.doi.org/#1}{\texttt{#1}}}

\end{document}